\documentclass[11pt]{article}

\usepackage{amsfonts}
\usepackage{mdframed}
\usepackage{amsmath}
\usepackage{amsthm}
\usepackage{fancyhdr}
\usepackage{enumerate}
\usepackage{graphicx}
\usepackage{color}
\usepackage{multicol}
\usepackage{pgfplots}
\usepackage{array}
\usepackage{tasks}
\usepackage{svg}
\usepackage{float}
\usepackage{caption}
\usepackage{subcaption}

\usepgfplotslibrary{fillbetween}
\tikzset{
  jumpdot/.style={mark=*,solid},
  excl/.append style={jumpdot,fill=white},
  incl/.append style={jumpdot,fill=black},
}

\newcommand{\N}{\mathbb{N}}

\newcommand{\C}{\mathbb{C}}

\newcommand{\Ulambda}{U_\lambda}
\newcommand{\Arg}{\operatorname{Arg}}

\newcommand{\lcu}{\left\lbrace}
\newcommand{\rcu}{\right\rbrace}

\newcommand{\lpa}{\left(}
\newcommand{\rpa}{\right)}

\newtheorem{theorem}{Theorem}
\newtheorem{proposition}{Proposition}
\newtheorem{manualpropositioninner}{Proposition}
\newenvironment{manualproposition}[1]{%
    \renewcommand\themanualpropositioninner{#1}%
  \manualpropositioninner
}{\endmanualpropositioninner}

\newtheorem{manuallemmainner}{Lemma}
\newenvironment{manuallemma}[1]{%
    \renewcommand\themanuallemmainner{#1}%
  \manuallemmainner
}{\endmanuallemmainner}

{\theoremstyle{definition}
\newtheorem{definition}{Definition}
}

\pgfplotsset{compat=1.18}
\begin{document}
\title{Resolving an error with path-tracing and a 2-to-1 mapping in a work of Jang, So, and Marotta}
\author{Murali Meyer, Daniel Stoertz, Mike Wang}
\date{\today}
\maketitle

\begin{abstract}
We identify two nontrivial errors in the proof of the main results of the work \emph{generalized baby Mandelbrot sets adorned with halos in families of rational maps} \cite{JSM} by Jang, So, and Marotta. We correct these errors, showing that the main results remain true.
\end{abstract}

\section{Introduction}\label{sec:intro}
The work \cite{JSM} of Jang, So, and Marotta aims to show, among other things, that for the family of rational maps on the Riemann sphere given by $F_{\lambda}(z)=z^n+\frac{\lambda}{z^d}$ where $n,d \in \N$ with $\frac{1}{n}+\frac{1}{d} < 1$, the variable $z \in \C$, and the parameter $\lambda \in \C$, there exist $n-1$ baby Mandelbrot sets in the connectedness locus for the filled Julia set of $F_{\lambda}(z)$. This is done by application of the Douady-Hubbard theory of polynomial-like maps established in \cite{DH}. In this note, we show that the proofs of two crucial results from \cite{JSM} contain nontrivial errors. Here we further give corrections to the constructions underlying the incorrect arguments, thereby proving that the main results of \cite{JSM} are true.

The errors involve two necessary conditions for the application of the Douady-Hubbard theory. One of these conditions that the family of maps $F_{\lambda}$ is polynomial-like of degree 2 on a region $U_{\lambda}'$ onto another region $U_\lambda$, with $\lambda$ in a specified region called $W$ in the parameter plane. Another condition of \cite{DH} is that, as $\lambda$ winds once around $\partial W$, the critical value of $F_\lambda$ must wind once around the exterior of $\Ulambda'$ while remaining in $\Ulambda$. Notably, both $\Ulambda$ and $\Ulambda'$ must be open sets in order to apply the result of \cite{DH}.

A brief summary of the first error is as follows. Proposition 5.1 in \cite{JSM} lays out the proof of $F_{\lambda}$ being polynomial-like of degree 2. One of the criteria is that $F_\lambda$ must map $\partial \Ulambda'$ 2-to-1 onto $\partial \Ulambda$. The constructed region $U_\lambda'$ is roughly a polar rectangle, where its boundary naturally splits into four segments, these being two ray segments and two non-circular arcs. The $\Ulambda$ region informally has a Pac-Man shape, its boundary consisting of a radial arc and two ray segments emanating from the origin (see Figure \ref{fig:original} in Section \ref{sec:errors}). The authors claim to show that both ray segments of $\partial \Ulambda'$ map 2-to-1 onto the linear segments of $\partial \Ulambda$. We show that such a mapping is impossible and that the actual image of the ray segments of $\partial \Ulambda'$ is instead a set of two line segments that pass through and intersect at the origin. Here each ray segment of $\partial U_\lambda '$ maps 1-to-1 onto distinct line segments through the origin, and so the mapping of $U_\lambda '$ onto $U_\lambda$ by $F_\lambda$ cannot be 2-to-1 on all of $U_\lambda '$. We then propose a solution to this problem that leaves the main result of Proposition 5.1 intact. This is accomplished by restricting the argument bounds of $\Ulambda$.

The second error occurs in Lemma 5.2. The proof traces the path of the critical value $v_\lambda$ as $\lambda$ winds once around $\partial W$. For one of the segments of $\partial W$, the authors attempt to show that the winding condition is upheld by demonstrating that the critical value traces part of the boundary of $\Ulambda$. However, because $\Ulambda$ is an open set, this violates the condition that the critical value lie in $\Ulambda\setminus \Ulambda'$. We resolve this by redefining $\Ulambda$ as having an outer radius equal to the old radius of $\Ulambda$ plus some sufficiently small $\epsilon$.

It is worth noting that these errors specifically arise when $n\neq d$, as is the general case studied in \cite{JSM}. In the special case that $n=d$, the aforementioned results were proved by Devaney in \cite{D}. Additionally worth noting are more recent works by Boyd and Mitchell (\cite{BM}) and Boyd and Hoeppner (\cite{BH}), which study a related family of rational maps defined by $R_{n,a,c}(z) = z^n +\frac{a}{z^n} +c$ and also produce results similar to \cite{JSM}. Comparing parameters, we may say that $n=d$ in both \cite{BM} and \cite{BH}, and therefore these works are not subject to the same concerns as \cite{JSM}. Given all of this, we will assume throughout that $n\neq d$. 

This note is organized as follows. In Section \ref{sec:prelim}, we cite all results from \cite{DH} and \cite{JSM} required to understand the errors. In Section \ref{sec:errors}, we describe the errors in the proofs of Proposition 5.1 and Lemma 5.2. Finally in Section \ref{sec:resolution}, we define a new region $\widehat{U}_\lambda$ with an extended outer radius and restricted angle bounds compared to $U_\lambda$, and we define the region $\widehat{U}_\lambda'$ as its preimage. We then prove that Propostition 5.1 and Lemma 5.2 are true under these conditions.

\section{Preliminaries}\label{sec:prelim}

For convenience, we collect here all results from \cite{JSM} regarding the family $F_\lambda$ that are correct and needed to understand Proposition 5.1 and Lemma 5.2. We also recall some results from \cite{DH}.

Let $n,d\in \N$ satisfy $\frac{1}{n} + \frac{1}{d}<1$, and let $m=n+d$ denote the degree of $F_\lambda$, and let $\nu = e^{i\frac{2\pi}{m}}$ be a primitive $m$th root of unity. The family $F_\lambda$ has $2(m) -2$ critical points counted with multiplicity. Here $\infty$ has multiplicity $n-1$ and $0$ has multiplicity $d-1$. The remaining $m$ critical points depend on $\lambda$. Writing $\lambda = |\lambda|e^{i\psi}$ where $\psi = \Arg(\lambda)$, then the `free' critical points $c_j$ of $F_\lambda$ are given by $c_j = c_\lambda \nu^j$, where
	\[c_\lambda = \lpa \frac{d}{n}|\lambda| \rpa^{\frac{1}{m}} e^{i\frac{\psi}{m}} \]
and $j=0,1,\ldots,m-1$. All these critical points lie on a circle of radius $|c_\lambda|$. The corresponding critical values are of the form $v_j = v_\lambda \nu^{gj}$ where $g=m/\gcd(n,d)$, $j = 0,1,\ldots, \gcd(n,d)$, and
\[ v_\lambda = \frac{m}{d}\lpa \frac{d}{n}|\lambda| \rpa^{\frac{n}{m}}e^{i\frac{n\psi}{m}}.\]
Note that there are $m/\gcd(n,d)$ distinct critical values, each being mapped to by $\gcd(n,d)$ distinct critical points under $F_\lambda$.

The family $F_\lambda$ also has $m$ prepoles, i.e. preimages of $0$, of the form $p_j = p_\lambda \nu^j$ where $p_\lambda = (-\lambda)^{1/m}$ and $j=0,1,\ldots,m-1$. These lie on the same circle as the free critical points.

It is easy to check that $F_\lambda (\nu z) = \nu^n F_\lambda(z)$, and so all orbits behave symetrically under $F_\lambda$. For a more precise statement of this symmetry, see \cite[Theorem 2.1]{JSM}. The point at $\infty$ is a superattracting fixed point for any value of $\lambda$. Let $B_\lambda$ denote the immediate attracting basin of $\infty$, and let $T_\lambda$ be the component of $F_\lambda^{-1}(B_\lambda)$ that contains the origin. We call $T_\lambda$ the \emph{trap door}.

The \emph{connectedness locus} $\mathcal{C}$ of $F_\lambda$ is the set of $\lambda$-values for which the Julia set of $F_\lambda$ is connected. A primary concern of \cite{JSM} is to establish the existence of $n-1$ small copies of the Mandelbrot set, often referred to as \emph{baby Mandelbrot sets}, in $\mathcal{C}$. To locate these baby Mandelbrot sets, the authors rely on the Douady-Hubbard theory of polynomial-like maps established in \cite{DH}.

\begin{definition}\cite{DH}\label{def:polynomial-like} A map $G : U' \to G(U') = U$ is \emph{polynomial-like} if
\begin{itemize}
    \item $G' $ and $ G$ are bounded, open, simply connected subsets of $\C$,
    \item $U'$ is relatively compact in $U$, and
    \item $G$ is analytic and proper.
\end{itemize}
Further $G$ is polynomial-like of \emph{degree two} if $G$ is a 2-to-1 mapping except at finitely many points, and $U'$ contains a unique critical point of $G$.
\end{definition}

Douady and Hubbard give sufficient conditions for the appearance of baby Mandelbrot sets in the connectedness locus of a family of polynomial-like maps of degree two. The following theorem does not appear in a single statement in \cite{DH}, but is summarized well in a recent work of Boyd and Mitchell \cite{BM}.

\begin{theorem}[\cite{DH}, see \cite{BM}, Theorem 2.3]\label{thm:mandelbrot}
Assume we are given a family of polynomial-like maps $G_\lambda:\Ulambda'\to \Ulambda$ that satisfies the following.
	\begin{itemize}
	\item $\lambda$ is in an open set in $\C$ which contains a closed disk $W$.
	\item The boundaries of $\Ulambda'$ and $\Ulambda$ vary analytically as $\lambda$ varies.
	\item The map $(\lambda,z) \mapsto G_\lambda(z)$ depends analytically on both $\lambda$ and $z$.
	\item Each $G_\lambda$ is polynomial-like of degree two with a unique critical point $C_\lambda$ in $\Ulambda'$.
	\end{itemize}
Suppose for all $\lambda \in \partial W$ that the critical value $G_\lambda(C_\lambda) \in \Ulambda - \Ulambda'$ and that $G_\lambda(C_\lambda)$ makes a closed loop around the outside of $\Ulambda'$ as $\lambda$ winds once around $\partial W$.

If all this occurs, then the set of $\lambda$-values for which the orbit of $C_\lambda$ does not escape from $\Ulambda'$ is homeomorphic to the Mandelbrot set.
\end{theorem}

\section{The errors in \cite{JSM}}\label{sec:errors}

Towards the application of Theorem \ref{thm:mandelbrot}, the authors of \cite{JSM} construct the following sets. See Section 4 of \cite{JSM} for the reasoning behind all the specific values.

The region $W$ is defined as the polar rectangle of $\lambda$-values in the right half-plane enclosed by the arcs of the circles given by 
    \[|\lambda|=\frac{n}{d}\left(\frac{d}{2m}\right)^{\frac{md}{nd-m}} \text{and } |\lambda|=\frac{n}{d}\lpa\frac{2d}{m}\rpa^{\frac{m}{n}}\]
and by portions of the rays
    \[\Arg(\lambda)=\pm\frac{\pi}{n-1}.\]
    
To define $\Ulambda'$, first $\mu$ and $\Gamma$ are defined as the circles 
    \[\mu : |z|=\frac{1}{2}\lpa\frac{d|\lambda|}{n}\rpa^{\frac{1}{d}}\quad \text{and }\quad
    \Gamma : |z|=2.\] 
Then let $\gamma_d$ be the preimage of $\Gamma$ that lies outside $\mu$ in the trapdoor and $\gamma_n$ be the preimage of $\Gamma$ that lies inside $\Gamma$ in the attracting basin of infinity. The sector $\Ulambda'$ is then defined as the points in the open region bounded by arcs of the two simple closed curves $\gamma_d \text{ and } \gamma_n$ and portions of the rays
    \[\text{Arg }z=\frac{\psi\pm\pi}{m}.\]
Thus $\Ulambda'$ is roughly a polar rectangle, with non-circular arcs as radial boundaries.

Finally, $\Ulambda$ is defined as the image of $\Ulambda'$ under $F_\lambda$. See Figure \ref{fig:original}. 
	\begin{figure}[t]
    	\centering
    	\includegraphics[width=3in]{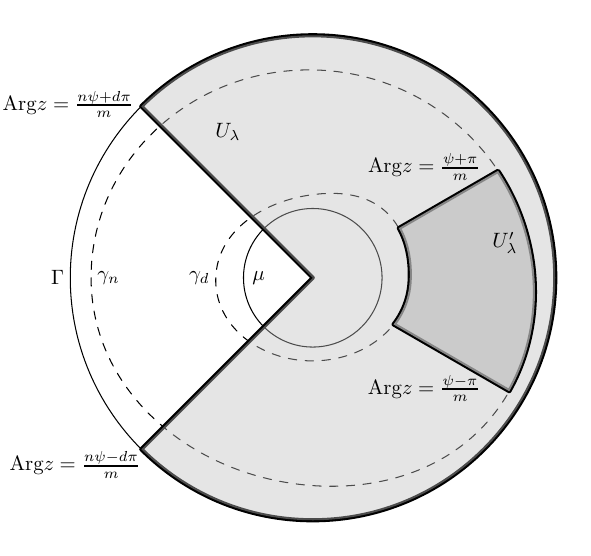}
    	\caption{Corrected Figure from \cite{JSM}. The sets $\Ulambda' \text{and } \Ulambda$, where $\Ulambda' \subset \Ulambda$. Here $\Ulambda$ is the lightly shaded Pac-Man shape, and $\Ulambda'$ is the darker shaded region similar to a polar rectangle. The circle $\Gamma$ lies in the attracting basin of infinity and the circle $\mu$ lies in the trapdoor.}
    	\label{fig:original}
	\end{figure}
As a quick note, we believe that the illustration corresponding to Figure \ref{fig:original}, which is \cite[Figure 5]{JSM}, contains a minor error and should have had $\gamma_d$ lying in the annulus bounded by $\Gamma$ and $\mu$. This is confirmed in the proof of \cite[Proposition 4.2]{JSM} where, if $\lambda\in W \text{ and } |z|\leq\frac{1}{2}\lpa\frac{d}{n}|\lambda|\rpa^{\frac{1}{d}}=\mu \text{ then } |F_\lambda(z)|\geq 2=\Gamma$. In other words, everything inside of $\mu$ is mapped outside of $\Gamma$ therefore the preimage of $\Gamma$ that lies in the trapdoor must be outside of $\mu$.

The first result from \cite{JSM} whose proof is incorrect is the following.

\begin{manualproposition}{5.1}\label{prop:5.1}
The family of maps $F_\lambda$ defined on $\Ulambda'$ with $\lambda\in W$ is a polynomial-like family of degree two.
\end{manualproposition}

In the proof of this proposition, the authors claim that the ray-segments of $\partial U_{\lambda}'$ each get mapped in a 2-to-1 fashion onto the lines passing through the origin with argument $\frac{n\psi \pm d\pi}{m}$, which form the piecewise linear segment of $\partial \Ulambda$. By construction, the ray-segments of $\partial\Ulambda'$ each contain a prepole of $F_{\lambda}$ and therefore must be mapped to line segments passing through the origin. However, for this mapping to be 2-to-1, each straight line segment of $\partial U_\lambda '$ would need to map onto the union of the piecewise linear segments of $\partial\Ulambda$ shown in Figure \ref{fig:original}. 

Yet $F_{\lambda}$ is a conformal mapping away from its critical points and poles, and one of the properties of conformal mappings is that they locally preserve angles between intersecting smooth curves. Recall that we are assuming that $n\neq d$, in which case the piecewise linear segments of $\partial\Ulambda$ are subsets of the rays
	\[\Arg z = \frac{n\psi \pm d\pi}{m}. \]
When $n\neq d$, these rays are not parallel at the origin, and hence the ray-segments of $\partial\Ulambda'$ would have to pass through the prepoles at an angle of $(2d\pi)/m \neq \pi$. Since this is not the case, we conclude that the mapping as stated cannot be 2-to-1.

Instead, the actual image of the union of the ray-segments of $\partial \Ulambda'$ is a union of two line segments crossing at the origin and intersecting nowhere else. Therefore, $F_\lambda$ maps parts of $U_{\lambda}'$ only 1-to-1 onto parts of $\Ulambda$, and the mapping is only 2-to-1 on a subset of $\Ulambda'$. See Figure \ref{fig:pacman2.0} for an illustration of this behavior. As a result, the maps $F_{\lambda}$ are not polynomial-like of degree two on $\Ulambda'$.

\begin{figure}[h]
        \centering
        \includegraphics[width=1\textwidth]{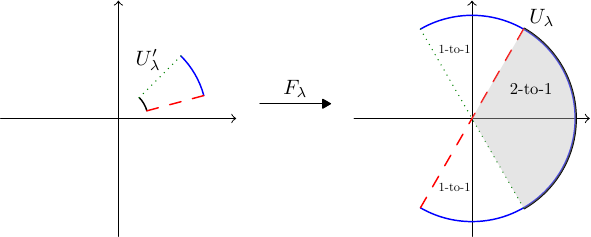}
        \caption{A color-coded picture of manner in which $\Ulambda'$ is mapped onto $\Ulambda$ by $F_\lambda$. Segments of $\partial U_\lambda'$ with a given color and texture are mapped to the corresponding segments of $\partial U_\lambda$. The overall mapping is only 2-to-1 onto the shaded subset of $U_\lambda$.}
        \label{fig:pacman2.0}
\end{figure}

The second result with an incorrect proof is the following. Note that this error is less severe than the first, and is accordingly easier to correct.

\begin{manuallemma}{5.2}[\cite{JSM}] \label{lem:5.2}
	There exists a small copy of the Mandelbrot set in the parameter $\lambda$-plane for $F_\lambda$ in each of the $n-1$ sectors of the form
		\[\frac{(2j-1)\pi}{n-1} < \Arg \lambda < \frac{(2j+1)\pi}{n}, \quad j = 0,1,2,\ldots, n-2. \]
\end{manuallemma}

In the proof of this lemma, the four segments of $\partial W$ are examined individually and the location of the critical value $v_\lambda$ is tracked. This is to fulfill the winding condition in the hypotheses of Theorem \ref{thm:mandelbrot}. The error occurs when $\lambda$ lies on the outer circular boundary of $W$, where
	\[|\lambda| = \frac{n}{d}\lpa \frac{2d}{m}\rpa^\frac{m}{n}. \]
In general, a straightforward computation shows that the modulus of $v_\lambda$ is
	\[|v_\lambda| = \frac{m}{d}\lpa\frac{d}{n}|\lambda|\rpa^{\frac{n}{m}}. \]
With $\lambda$ as above, this simplifies to $|v_\lambda| = 2$. So in this case, $v_\lambda$ lies on $\Gamma$, which is part of the boundary of $\Ulambda$. This violates the condition of Theorem \ref{thm:mandelbrot} that $v_\lambda$ must lie in $\Ulambda - \Ulambda'$, since $\Ulambda$ does not contain its boundary by construction. This cannot be fixed by simply taking closures, since Definition \ref{def:polynomial-like} requires $\Ulambda$ to be an open set. Hence Theorem \ref{thm:mandelbrot} cannot be applied.

Ultimately, the problems in both proofs are due to an incorrect setup when defining $\Ulambda'$ and $\Ulambda$.

\section{A resolution}\label{sec:resolution}

To fix the problems and retain the results from \cite{JSM}, we modify the authors' construction. Recall that we are assuming that $n\neq d$, both since the problems do not arise when $n=d$ and since Devaney already covered the $n=d$ case in \cite{D}.

Recall from Section \ref{sec:errors} the definitions of $\Gamma$ as the circle centered at origin with radius 2, $\mu$ as the circle centered at the origin with radius $\frac12(\frac{d}{n}|\lambda|)^{\frac{1}{d}}$, and $\gamma_d,\gamma_n$ as the preimages of $\Gamma$ under $F_\lambda$ such that $\gamma_d$ is in the trapdoor $T_\lambda$ and $\gamma_n$ is in the escaping basin $B_{\lambda}$.
  
For compact sets $\alpha$ and $\beta$, let $d(\alpha,\beta)$ denote the minimal Euclidean distance between any two points  $a\in\alpha$ and $b\in\beta$. By compactness, there exist $\delta_1,\delta_2>0$ such that $d(\Gamma,\gamma_n)>\delta_1$, and $d(\mu,\gamma_d)>\delta_2$.
  
For $\varepsilon>0$, let $\widehat\Gamma$ be the circle centered at origin with radius $2+\varepsilon.$ Let $\widehat\gamma_d,\widehat\gamma_n$ be the preimages of $\widehat\Gamma$ under $F_\lambda$ such that $\widehat\gamma_n\in B_{\lambda}$ and $\widehat\gamma_d\in T_{\lambda},$ which exist since $\widehat\Gamma\in B_\lambda$. Since all the closed curves and circles above are away from the pole at the origin, we can and do choose $\varepsilon$ small enough so that $d(\widehat\gamma_n,\gamma_n)<\frac12\min{(\delta_1,\delta_2)}$ and $d(\widehat\gamma_d,\gamma_d)<\frac12\min{(\delta_1,\delta_2)}$.

We then define $V_\lambda'$ as the region bounded by two rays of arguments $\frac{\psi\pm\pi}{m}$ and the two curves $\widehat\gamma_n$, $\widehat\gamma_d$. Then let $V_{\lambda}=F_\lambda(V_{\lambda}')$. Note that $\widehat\gamma_n$ and $\widehat\gamma_d$ map to different segments of $\widehat\Gamma$ since $n\neq d$ and the two rays map to two non-parallel line segments crossing the origin. So $V_{\lambda}$ has the same shape as the old $U_\lambda$, except for possibly the details on the outer and inner boundary.  Notice also that the arguments of these two ray-segments bounding $V_\lambda$ are $\frac{n\psi-d\pi}{m}$ and $\frac{n\psi+d\pi}{m}$. 

Define $\widehat U_{\lambda}\subset V_{\lambda}$ as the region bounded by $\widehat\Gamma$ and the two line segments above. So \[ \widehat U_{\lambda} = \lcu z\in\C \,:\, |z| < 2+\varepsilon,\, \theta_2 < \operatorname{Arg}(z) < \theta_1 \rcu \] where \[\theta_1=\min \left\{\frac{n\psi+d\pi}{m},\frac{n\psi-d\pi}{m}+\pi \right\},\ \ \ 
    \theta_2=\max \left\{ \frac{n\psi-d\pi}{m},\frac{n\psi+d\pi}{m}-\pi \right\}.\]
Note that $\widehat U_\lambda$ has a Pac-Man shape similar to $U_\lambda$ as shown in Figure \ref{fig:pacman2.0}. We then define $\widehat U_{\lambda}'$ as the component of the preimage of $\widehat U_{\lambda}$ under $F_\lambda$ containing a segment of the ray $\Arg z = \psi/m$. See Figure \ref{fig:z-plane}.

  \begin{figure}[h]
    \centering
    \includegraphics[width=0.7\textwidth]{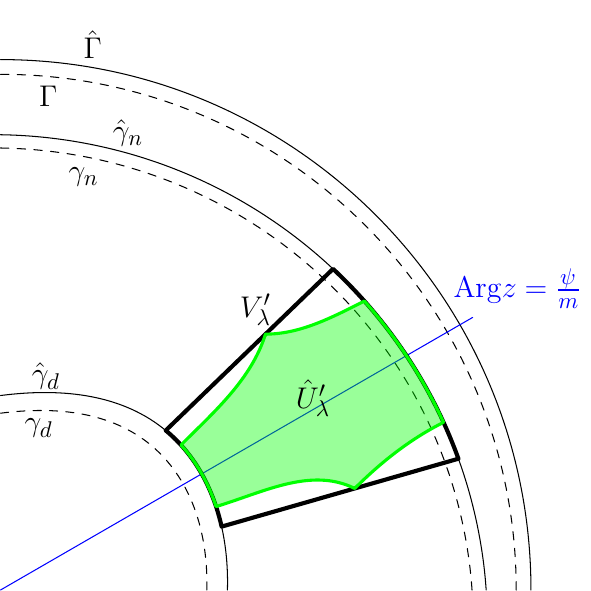}
    \caption{Constructing $\widehat U_\lambda '$. Each of the curves $\gamma_d$, $\gamma_n$, and $\Gamma$ used in the construction of the original $U_\lambda$ is dashed. The new curves $\widehat\gamma_d$, $\widehat \gamma_n$, and $\widehat\Gamma$ are solid. $V_\lambda'$ is constructed as detailed above, then $V_\lambda$ (not depicted) is defined as $F(V_\lambda ')$. We take a subset $\widehat U_\lambda$ (not depicted) of $V_\lambda$ as described above, then define $\widehat U_\lambda '$ to be a well-chosen component of the preimage of $U_\lambda$ under $F$.}\label{fig:z-plane}
\end{figure}

\begin{proposition} \label{prop:subset}
 $V_{\lambda}'$ is a subset of $\widehat U_{\lambda}$.
\end{proposition}
    \begin{proof}
    	Since the outer boundary of $V_\lambda'$ lies on $\widehat\gamma_n$, the maximum modulus of any point in $V_\lambda'$ is less that $2+\varepsilon$, which is the supremum of the moduli of points in $\widehat U_\lambda$. So it remains to show that $\theta_2<\frac{\psi-\pi}{m}$ and $\frac{\psi+\pi}{m}<\theta_1$.

        We will consider two cases: $n>d$, and $n<d.$
    
    \noindent {\bf Case 1:} $n>d$. 
    
    	By definition of $\theta_1$ and $\theta_2$, we have
            \[\theta_1=\frac{n\psi+d\pi}{m},\quad \text{and} \quad \theta_2=\frac{n\psi-d\pi}{m}.\]
        Since $n,d>1$, we have        
            \[\frac{\psi+\pi}{m}<\frac{n\psi+d\pi}{m}=\theta_1.\] 
        Recall by definition of $W$ that $|\psi|\leq \pi/(n-1)$. It follows that $(n-1)\psi<(d-1)\pi$, and so $n\psi-d\pi<\psi-\pi.$ Hence
            \[\frac{\psi-\pi}{m}>\frac{n\psi-d\pi}{m}=\theta_2.\]
        Therefore both inequalities hold when $n>d.$

	\noindent {\bf Case 2:} $n<d$.
	
		Again by definition of $\theta_1$ and $\theta_2$, we now have
            \[\theta_1=\frac{n\psi-d\pi}{m}+\pi=\frac{n\psi+n\pi}{m}, \quad \text{and} \quad \theta_2=\frac{n\psi+d\pi}{m}-\pi=\frac{n\psi-n\pi}{m}.\] 
        Clearly,
            \[\frac{\psi+\pi}{m}<\frac{n\psi+n\pi}{m}=\theta_1.\]
        Moreover, since $\psi\leq\frac{\pi}{n-1}<\pi$, we have $\psi-\pi<0$. It then follows that 
            \[\frac{\psi-\pi}{m}>n\frac{\psi-\pi}{m}=\frac{n\psi-n\pi}{m}=\theta_2.\] 
        Therefore, both inequalities hold when $n<d.$
        Thus, $V_{\lambda}'\subset \widehat U_{\lambda}.$
    \end{proof}
  
    Since $\widehat U_{\lambda}\subset V_{\lambda}$, we have by choice of component that $\widehat U_{\lambda}'\subset V_{\lambda}'$. By Proposition \ref{prop:subset} we have $\widehat U_{\lambda}'\subset \widehat U_{\lambda}$. Furthermore, since $\widehat U_\lambda$ is the subset of $V_\lambda$ onto which $F_\lambda$ maps $V_\lambda '$ in a 2-to-1 fashion, we can conclude that the function $F_\lambda:\ U_{\lambda}'\to U_{\lambda}$ is polynomial-like of degree two. This proves the following replacement of Proposition \ref{prop:5.1}.
    
    \begin{proposition} \label{prop:2to1}
    Let $\widehat U_\lambda'$ be defined as above. Then the family $F_\lambda$ defined on $\widehat U_\lambda'$ with $\lambda \in W$ is a polynomial-like family of degree two.
    \end{proposition}
  
    It remains to show that the winding condition of Theorem \ref{thm:mandelbrot} is satisfied. 
    
    \begin{proposition} \label{prop:winding}
    For all $\lambda \in \partial W$, the critical value $v_\lambda$ lies in $\widehat U_\lambda - \widehat U_\lambda'$. Moreover, as $\lambda$ winds once around $\partial W$, the critical value $v_\lambda$ makes a closed loop around the outside of $\widehat U_\lambda'$.
    \end{proposition}
    \begin{proof}
    	Note that the set $W$ remains the same as in \cite{JSM}, so as $\lambda$ winds around $\partial W$, the path of the critical value $v_\lambda$ remains the same as in the proof of Lemma \ref{lem:5.2}. Hence $v_\lambda$ follows the two curves $\mu$ and $\Gamma$ and the two rays with arguments $\frac{\psi\pm\pi}{m}$.
    	
    	If $v_\lambda$ lies on $\Gamma$, then by construction of $\widehat\gamma_n$, we conclude that $v_\lambda \in \widehat U_\lambda - \widehat U_\lambda'$. Similarly, if $\lambda$ lies on $\mu$, then by construction of $\widehat\gamma_d$, we conclude that $v_\lambda \in \widehat U_\lambda - \widehat U_\lambda'$.
    	
    	If $\Arg \lambda = \pi/(n-1)$, then for any $z \in \widehat U_\lambda '$ we have
    		\[\Arg z \leq \frac{\psi+\pi}{m} = \frac{n\pi}{m(n-1)} = \lpa \frac{\pi}{n-1}\rpa\frac{n}{m} = \Arg v_\lambda.\]
    	Then $v_\lambda \in \widehat U_\lambda - \widehat U_\lambda'$. Similarly, if $\Arg \lambda = -\pi/(n-1)$ then for any $z \in \widehat U_\lambda '$ we have
	   		\[\Arg z \geq \frac{\psi-\pi}{m} = \frac{-n\pi}{m(n-1)} = \Arg v_\lambda.\]
	   		
	   	Hence the winding condition of Theorem \ref{thm:mandelbrot} is satisfied.
    \end{proof}
    
    This proposition proves that the original statement of Lemma 5.2 in \cite{JSM} is true.

\end{document}